\documentclass[conference,a4paper]{IEEEtran}

\usepackage[latin1]{inputenc}
\usepackage[T1]{fontenc}

\usepackage{balance}

\usepackage[cmex10]{amsmath}
\usepackage{amssymb,amsthm}
\usepackage{mathtools}
\newtheorem{thm}{Theorem}
\newtheorem{proposition}{Proposition}

\newtheorem{definition}[proposition]{Definition}

\usepackage{enumerate}

\usepackage[babel]{csquotes}

\usepackage{tikz}

\newcommand{\R}{\ensuremath{\mathbb{R}}}

\newcommand{\N}{\ensuremath{\mathbb{N}}}

\newcommand{\F}{\ensuremath{\mathcal{F}}}
\newcommand{\G}{\ensuremath{\mathcal{G}}}
\renewcommand{\S}{\ensuremath{\mathcal{S}}}
\renewcommand{\H}{\ensuremath{\mathcal{H}}}
\renewcommand{\G}{\ensuremath{\mathcal{G}}}

\newcommand{\PW}{\ensuremath{\mathcal{PW}}}
\newcommand{\BP}{\ensuremath{\mathcal{BP}}}

\newcommand{\Radon}{\ensuremath{\mathcal{R}}}

\newcommand{\dx}[1]{\ensuremath{\,d\kern-0.05em#1}}
\newcommand{\dxx}[2]{\ensuremath{\,d\kern-0.05em#1}\kern-0.05em(\kern-0.075em#2\kern-0.075em)}
\newcommand{\intsw}{\ensuremath{\int_{S^1}\hspace{-0.5em}\dx{\omega}}}

\newcommand{\st}{\ensuremath{\text{s.t.}}}

\newcommand{\textdef}{\textnormal}

\DeclareMathOperator{\id}{id}

\DeclareMathOperator{\sinc}{sinc}

\begin{document}
%
\title{Irregular Sampling of the Radon Transform of Bandlimited Functions}

\author{\IEEEauthorblockN{Thomas Wiese}
\IEEEauthorblockA{Associate Institute for Signal Processing\\
Technische Universität München, Germany\\
Email: thomas.wiese@tum.de}
\and
\IEEEauthorblockN{Laurent Demaret}
\IEEEauthorblockA{Helmholtz Center Munich, Germany\\
Email: laurent.demaret@helmholtz-muenchen.de}
}

\maketitle
\begin{abstract}
We provide conditions for exact reconstruction of a bandlimited function from irregular polar samples of its Radon transform. First, we prove that the Radon transform is a continuous $L^2$-operator for certain classes of bandlimited signals. We then show that the Beurling-Malliavin condition for the radial sampling density ensures existence and uniqueness of a solution. Moreover, Jaffard's density condition is sufficient for stable reconstruction.
\end{abstract}

\IEEEpeerreviewmaketitle

\section{Introduction}
In computed tomography (CT), a central question is the following: what kind of detail can be resolved from a particular CT scan (sinogram)? Notwithstanding the lack of a clear definition of the term detail, in many applications a satisfactory and useful answer is provided in terms of the Nyquist frequency connected to the sampling geometry. 

The \enquote{pure} case of reconstructing a function from its possibly irregular samples has been solved nicely for classes of bandlimited functions in terms of density theorems as we will see in Section~\ref{sec:review} of this paper. Due to difficulties that arise when defining the Radon transform for bandlimited functions, these results have not yet been used in the context of CT. Instead, efforts revolved around the analysis of quasi-bandlimited functions and the results bear the deficiency of only asymptotically controllable errors~\cite{Natterer86,Rattey81}.

This paper closes this apparent gap in the literature, namely, we show that the Radon transform can be defined in the usual sense as a continuous $L^2$-operator for certain classes of bandlimited functions. The Radon transform of such signals is itself bandlimited and it is shown that stable and exact reconstruction of these signals from their irregularly sampled Radon transforms is possible if the sampling set satisfies certain density requirements.

The remainder of this paper is organized as follows: In Section~\ref{sec:motivation} we present current techniques in reconstruction of the sampled Radon transform and how they relate to spaces of bandlimited functions and sampling. In Section~\ref{sec:review} we provide a dense overview of results from sampling theory for bandlimited functions. After showing continuity of the Radon transform and its inverse for certain bandlimited functions in Section~\ref{sec:radonpw}, we will apply these results to the sampled Radon transform in Section~\ref{sec:radonsampling}.

\section{Motivation and Related Work}\label{sec:motivation}
For efficient experimental design of CT scans, i.e. determination of a suitable sampling geometry or a posteriori choice of function spaces for reconstruction, it is essential to understand the discretization effects due to sampling. Furthermore, the emergence of CT acquisition procedures involving incomplete or irregular data calls for irregular sampling theory. Past research has focused
on functions that are simultaneously essentially space- and band-limited, i.e., functions for which
\begin{equation*}
\int_{\R^2 \setminus B_R} |f(x)|^2 \dx{x} \quad \text{and} \quad 
\int_{\R^2 \setminus B_R} |(\F f)(\xi)|^2 \dx{\xi}
\end{equation*}
decay exponentially with the radius $R$ of the ball $B_R$. For these functions, interleaved sampling geometries are more efficient than regular sampling geometries~\cite{Natterer86,Rattey81}.

Bandlimitedness conditions also appear implicitly in reconstruction techniques based on discretizations of the inverse Radon transform. Filtered backprojection tacitly assumes that the Radon transform is bandlimited and periodic in the radial coordinate (for computation of a so-called absolute derivative operator) and that quadrature rules for the angular integral (for the backprojection) are exact---for example by assuming that the angular component of the Radon transform has a finite Fourier series representation~\cite{Natterer86}.

Algorithms that are based on the Fourier slice theorem commonly use some sort of Fast Fourier Transform (FFT) for the radial variable to obtain the Fourier transform of the unknown function on a polar grid. This operation is either followed by interpolation onto a rectangular grid and application of the two-di\-men\-sion\-al inverse FFT---a process known as gridding~\cite{Epstein03,Margolis04}---or by using a version of the two-di\-men\-sion\-al inverse FFT for non-rectangular grids~\cite{Grinberg00,Potts02}. The assumptions are, again, that the Radon transform $\Radon f$ is bandlimited and periodic with respect to the radial coordinate and that $f$ is bandlimited and periodic in both Cartesian variables.

Finally, algebraic reconstruction techniques can handle all sorts of irregular grids and are very efficient. However, it is difficult to analyze irregular sampling with such methods. First, it is hard to find function spaces for which the system matrix is injective and second, the combined effects of regularization, noise reduction, and early termination of iterative solvers are hard to quantify and isolate from sampling effects.

\section{Review of Sampling Theory for Bandlimited Functions}\label{sec:review}
We provide a brief review of the available theorems and techniques used for irregular sampling of bandlimited square-integrable functions in one dimension. These results can be extended to more dimensions when sampling on product grids.
\begin{definition}[Paley-Wiener spaces]\label{def:paleywiener}
Let $\F$ denote the Fourier transform. The Paley-Wiener space of $R$-radially bandlimited and square-integrable functions is defined as
\begin{equation*}
\PW_R(\R^d) \coloneqq \{ f \in L^2(\R^d) : \F f|_{\R^d \setminus B_R^d} = 0\},
\end{equation*}
where $B_R^d$ is the $d$-dimensional ball with radius $R$. Similarly, for $r>0$, we define the space of bandpass functions as
\begin{equation*}
\BP_{rR}(\R^d) \coloneqq \{ f \in L^2(\R^d) : \F f|_{\R^d\setminus (B_R^d \setminus B_r^d)} = 0\}.
\end{equation*}
\end{definition}
Let $\Lambda \subset \R$ be a \textit{uniformly discrete} set of sample positions, that is, $\inf_{\lambda,\mu} |\lambda - \mu| > 0$ for $\lambda,\mu \in \Lambda$ and $\lambda \neq \mu$. This condition ensures that the sampling operator $S_\Lambda \colon \PW_R(\R) \to l^2(\Lambda)$ is always a continuous linear operator into $l^2(\Lambda)$~\cite{Groechenig92}. For a fixed bandwidth $R$, sampling theory gives conditions in terms of densities on the sampling set $\Lambda$ under which functions in $\PW_R(\R)$ can be identified by and reconstructed from its values on $\Lambda$. In particular, one wishes to establish whether~\cite{Benedetto01}
\begin{itemize}
\item $\Lambda$ is a \textit{set of uniqueness} for $\PW_R(\R)$, i.e., the sampling operator $S_\Lambda$ is injective or whether
\item $\Lambda$ is a \textit{set of sampling} for $\PW_R(\R)$, i.e., the sampling operator $S_\Lambda$ is continuous and continuously invertible on its range.
\end{itemize}
\begin{definition}[Densities]
Let $\Lambda \subset \R$ be uniformly discrete with $0 \notin \Lambda$ and with signed counting function $N_\Lambda(t)$, which counts the number of points in the interval with endpoints $0$ and $t$ and has negative sign for $t<0$.
\begin{enumerate}[i)]
\item The \textdef{Beurling-Malliavin density} is defined as
\begin{equation*}
D_{bm}(\Lambda) = \inf_{c \geq 0}c\, \;\st\; \left\{\begin{aligned}
&\exists\, h \in C^1(\R),\, 0\leq h^\prime(t) \leq c,\\
&\int_\R \frac{|N_\Lambda(t)-h(t)|}{1+t^2} \dx{t}< \infty
\end{aligned}\right\} .
\end{equation*}
\item The \textdef{frame density} is defined as
\begin{equation*}\label{eq:uniformdensity}
D_f(\Lambda)\coloneqq \sup_{\Gamma \subset \Lambda}\sup_{c \geq 0} \,c \;\, \st \; N_\Gamma(t)- ct = O(1),
\end{equation*}
where the supremum is over all subsets $\Gamma$ for which the asymptotics exist and $D_f(\Lambda)=0$ if no such subset exists.
\end{enumerate}
\end{definition}

These densities apply to reasonably general sampling sets. In particular, the frame density is invariant under removals of finitely many points, i.e., one arbitrarily sized \enquote{hole} is allowed. In case of the Beurling-Malliavin density, it is possible to construct grids with $D_{bm}(\Lambda)=1$ that have infinitely many \enquote{holes} of unbounded size~\cite{Levinson40}. If $\Lambda$ is a set for which there exists $c>0$ and for which the asymptotics $N_\Lambda(t)-ct=O(1)$ hold, one also says that $\Lambda$ has \textit{uniform density} $c$. Our definition of the Beurling-Malliavin density can be found in~\cite{Kahane61}. It is a simplification of the original \textit{exterior density} $A_e(\dx{N_\Lambda})$ used by Beurling and Malliavin~\cite{Beurling67}, which also applies for sampling sets of complex numbers. The following theorem encapsulates several decades of research~\cite{Beurling67,Beurling62,Jaffard91,Landau67}.
\begin{thm}[Sampling theorems]\label{thm:samplingtheorem}
For $\Lambda$ to be a set of uniqueness for $\PW_\pi(\R)$
\begin{enumerate}[(i)]
\item it is necessary that $D_{bm}(\Lambda) \geq 1$,
\item it is sufficient that $D_{bm}(\Lambda) > 1$.
\end{enumerate}
For $\Lambda$ to be a set of sampling for $\PW_\pi(\R)$
\begin{enumerate}[(i)]
\item it is necessary that $D_f(\Lambda) \geq 1$,
\item it is sufficient that $D_f(\Lambda) > 1$.
\end{enumerate}
\end{thm}
In the above theorem, it is possible to replace $\pi$ with $R>0$ and $1$ with $R/\pi$ on the right hand side of the density conditions. The last condition is known as Jaffard's sufficient condition.

Using the theory of frames, reproducing kernel Hilbert spaces (RKHS), and tensor products, one can generalize these results to two (and more) dimensions for Cartesian products of sampling grids~\cite{Benedetto01,Bourouihiya08}. A RKHS $\H$ with domain $\R^d$ is a Hilbert space in which all point evaluations are continuous linear functionals, i.e, for all $x \in \R^d$, the map $f \mapsto f(x)$ is a bounded linear functional in $\H$~\cite{Aronszajn50}. Paley-Wiener spaces and subspaces of $L^2(\R^d)$ spanned by a finite number of functions are examples of RKHS~\cite{Yao67a}. 
\begin{thm}\label{thm:tensorproduct}
Let $\H_1$ and $\H_2$ be RKHS of functions on $\R$. If for $i=1,2$, $\Lambda_i$ is a set of uniqueness, resp.\ set of sampling, for $\H_i$, then $\Lambda_1 \times \Lambda_2$ is a set of uniqueness, resp.\ set of sampling, for $\H_1 \otimes \H_2$.
\end{thm}
The result is a consequence of the fact that tensor products of complete systems are complete in the tensor product space and that tensor products of frames are frames in the tensor product space~\cite{Bourouihiya08}.

\section{Radon Transform of Bandlimited Functions}\label{sec:radonpw}
In an effort to apply Theorem~\ref{thm:samplingtheorem} to the irregularly sampled Radon transform, we first need to ensure \textit{compatibility} between Paley-Wiener spaces and the Radon transform.
Therefore, in this section, we establish conditions under which the Radon transform can be defined as a continuous and continuously invertible $L^2$-operator between subspaces of $\PW_R(\R^2)$ and $\PW_R(\R) \otimes L^2(S^1)$, where $S^1$ is the unit sphere in $\R^2$. Our approach contrasts with the conventional definition of the Radon transform as a continuous---but not continuously invertible---operator between $L^2(B_R(\R^2))$ and $L^2([-R,R] \times S^1)$. The advantage of our definition is that for irregular sampling grids of the form $\Lambda_s \times \Lambda_\omega$ with $\Lambda_s \subset \R$ and $\Lambda_\omega \subset [0,\pi]$, we can apply Theorems~\ref{thm:samplingtheorem} and~\ref{thm:tensorproduct} to find conditions under which \textit{exact} and \textit{stable} reconstruction of a function from its sampled Radon transform is possible.

First, we present a counter example which highlights the difficulties that arise when defining the Radon transform for bandlimited functions. We will use the well-known Fourier slice theorem, which provides the following decomposition of the Radon transform for Schwartz functions $f \in \S(\R^2)$:
\begin{equation*}\label{eq:radondecomposition}
(\Radon f)(s,\omega) = (\F_s^{-1}\id\Phi\F f)(s, \omega).
\end{equation*}
Here, we denote the two-dimensional Fourier transform by $\F$, the one-dimensional Fourier transform with respect to the radial coordinate by $\F_s$, and the change from polar to Cartesian coordinates by $\Phi\colon L^2(\R^2) \to L^2(\R^+ \times S^1, \sigma \dx{\sigma} \otimes \dx{\omega})$. By explicitly considering the \textit{change of norms}, 
\begin{equation*}
\id \colon L^2(\R^+ \times S^1, \sigma\dx{\sigma}\otimes \dx{\omega}) \to L^2(\R \times S^1, \dx{\sigma} \otimes \dx{\omega}),
\end{equation*}
we retain the $L^2$-isometry property (up to some power of $2\pi$) of the Fourier transforms and, hence, $L^2$-continuity of the overall operator is determined solely by that of the change of norms.\footnote{For $\sigma < 0$ we define $(\id g)(\sigma,\omega) = g(-\sigma,-\omega)$, which ensures that the new variables can still be interpreted as polar coordinates.} 
To show that the Radon transform is not $L^2$-continuous on $\PW_R(\R^2)$, consider the sequence $f_n \in \PW_R(\R^2)$ defined through its Fourier transform:
\begin{equation*}
(\F f_n)(\xi) = \begin{cases}
|\xi|^{-1/2}&\text{if}\; n^{-1} \leq |\xi| \leq R,\\
0&\text{otherwise.}
\end{cases}
\end{equation*}
Each $f_n$ is bandlimited, because $(\F f_n)(\xi) = 0$ for $|\xi| > R$, and $f_n \in L^2(\R^2)$, because $\F f_n \in L^2(\R^2)$ as
\begin{equation*}
\begin{aligned}
\int_{\R^2} \dx{\xi} |(\F f_n)(\xi)|^2
&= \intsw \int_{n^{-1}}^R \sigma \dx{\sigma} \left|\sigma^{-1/2}\right|^2\\
&= \intsw \int_{n^{-1}}^R \dx{\sigma}\\
&= 2\pi (R-n^{-1}).
\end{aligned}
\end{equation*}
As can be seen, the norm of $\F f_n$ tends to $\sqrt{2\pi R}$ and that of $f_n$ to $\sqrt{R/(2\pi)}$. On the other hand, the norm of $(\id \Phi \F)(f_n)$ with respect to the \enquote{flat} measure $\dx{\sigma}\otimes\dx{\omega}$ and, hence, that of $\Radon f_n$ in $L^2(\R\times S^1)$, is unbounded, because
\begin{equation*}
\begin{aligned}
\intsw \int_\R \dx{\sigma}|(\Phi\F f_n)(\sigma,\omega)|^2
&= 2\intsw \int_{n^{-1}}^R \dx{\sigma} \left|\sigma^{-1/2}\right|^2\\
&= 2\intsw \int_{n^{-1}}^R \dx{\sigma} \sigma^{-1}\\
&= 4\pi(\ln R + \ln n),
\end{aligned}
\end{equation*}
which tends to infinity as $n$ grows.\footnote{The factor $2$ is a consequence of $\id$ mapping $\R^+$ to the whole real line.} Thus, the Radon transform cannot be continuous on $\PW_R(\R^2)$.

This defect of the Radon transform is a consequence of the fact that the Fourier transforms of functions in $\PW_R(\R^2)$ may have mild singularities at the origin. If we restrict the Paley-Wiener space to bandpass functions $f \in \BP_{rR}(\R^2)$, we can easily verify the boundedness of the Radon transform:
\begin{equation*}\label{eq:boundedness}
\begin{aligned}
\intsw \int_\R \dx{\sigma}\left|(\F f)(\sigma,\omega)\right|^2
&= 2\intsw \int_r^R \hspace{-0.5em}\dx{\sigma} \left|(\F f)(\sigma,\omega)\right|^2\\
&\leq 2\intsw \int_r^R \hspace{-0.5em}\dx{\sigma}\frac{\sigma}{r} \left|(\F f)(\sigma,\omega)\right|^2\\
&= 2r^{-1} \int_{\R^2}\dx{\xi} |(\F f)(\xi)|^2.
\end{aligned}
\end{equation*}
The same calculation---replace $r$ with $R$ in the denominator and turn around the inequality---also yields the lower bound
\begin{equation*}
2R^{-1}\|f\|^2 \leq \left\| \Radon f \right\|^2 \leq 2r^{-1}\|f\|^2.
\end{equation*}
This implies (e.g.,~\cite{Rynne08}, Thm.~4.48) closedness of the range and existence of a continuous inverse of the Radon transform:
\begin{thm}[Radon isomorphism]\label{thm:radonisomorphism}
The Radon transform is a Hilbert space isomorphism between the Hilbert spaces $\BP_{rR}(\R^2)$ and $\Radon (\BP_{rR}(\R^2)) \subset L^2(\R \times S^1, \dx{s} \otimes \dx{\omega})$.
\end{thm}
We can also characterize the range of the Radon transform for bandpass functions. The theory of tensor products of separable $L^2$-spaces yields the decomposition~\cite{Reed81}
\begin{equation*}
L^2(\R \times S^1) \simeq L^2(\R,\dx{s}) \otimes L^2(S^1,\dx{\omega}).
\end{equation*}
The Fourier slice theorem shows that if $(\F f)(\xi) = 0$ for $|\xi|<r$ and $|\xi|>R$, then $(\F_s \Radon f)(\sigma,\omega)$ also vanishes for $\sigma$ outside of $[-R,-r]\cup [r,R]$. Hence, if $f \in \BP_{rR}(\R^2)$, then $\Radon f \in \BP_{rR}(\R)\otimes L^2(S^1)$. Similarly, for $g \in \BP_{rR}(\R) \otimes L^2(S^1)$ with $g(s,\omega)=g(-s,-\omega)$, we can go the inverse way of the Fourier slice theorem to define a function $f = \F^{-1} \Phi^{-1} \id^{-1} \F_s g$. The same calculations as above show that $f \in \BP_{rR}(\R^2)$ and since, by definition, $g = \Radon f$, we obtain:
\begin{thm}[Range theorem for bandpass functions]\label{thm:range}
The Radon transform maps $\BP_{rR}(\R^2)$ isomorphically to $\BP_{rR}(\R \times S^1)$, where we define
\begin{equation*}
\BP_{rR}(\R\times S^1) \coloneqq \left\{\begin{aligned} f \in \BP_{rR}(\R) \otimes L^2(S^1)\\\text{with}\; f(-s,-\omega) = f(s,\omega)\end{aligned}\right\}.
\end{equation*}
\end{thm}
Remark that for $g \in \BP_{rR}(\R \times S^1)$, the Helgason-Ludwig moment conditions~\cite{Helgason80} are automatically satisfied, because $\F_s g$ and all of its derivatives vanish around the origin:
\begin{equation*}
\int_\R g(s,\omega)s^k ds = \left(\frac{i}{2\pi}\right)^k \frac{d^k}{d\sigma^k} (\F_s g)(0,\omega) = 0\;\forall\, k \in \N_0.
\end{equation*}

\section{Sampling Theorems for the Radon Transform}\label{sec:radonsampling}
The developments from the previous section allow us to apply the theory of bandlimited functions to the sampled Radon transform. Because of the isomorphism property established in Theorems~\ref{thm:radonisomorphism} and~\ref{thm:range}, the sampled Radon transform operator $\Radon_\Lambda \colon \BP_{rR}(\R^2) \to l^2(\Lambda), f\mapsto ((\Radon f)(\lambda))_{\lambda \in \Lambda}$ is continuous and continuously invertible on its range exactly if the sampling operator $S_\Lambda \colon \BP_{rR}(\R \times S^1) \to l^2(\Lambda), g\mapsto (g(\lambda))_{\lambda \in \Lambda}$ is continuous and continuously invertible on its range; hence, we can concentrate our analysis on the latter.

One possible way of getting rid of the symmetry requirement is to consider sampling sets of the form $\Lambda \subset \R \times [0,\pi]$ and interpret the Radon transform as a map from $\BP_{rR}(\R^2)$ to $\BP_{rR}(\R) \otimes L^2([0,\pi])$. The inverse of the sampled Radon transform is then given as $\Radon^{-1} \mathcal{I} S_\Lambda^{-1}$, where $\mathcal{I} : \BP_{rR}(\R) \otimes L^2([0,\pi]) \to \BP_{rR}(\R \times S^1)$ is the isomorphism defined by
\begin{equation*}
(\mathcal{I}f)(s,\omega) = \begin{cases}
f(s,\arg(\omega)) & \text{if}\;0\leq \arg(\omega) < \pi,\\
f(-s,\arg(\omega)-\pi) &\text{otherwise}.
\end{cases}
\end{equation*}
It is equally possible, but slightly more technical, to allow for sampling grids $\Lambda \subset \R^+ \times [0,2\pi]$. However, due to space limitations, we will postpone comments on that case to an upcoming journal publication.

Lastly, we need to restrict the angular behavior of admitted functions to some finite dimensional and, thus, automatically RKHS subspace $\G \subset L^2([0,\pi])$. The finiteness condition is a consequence of $[0,\pi]$ being bounded. One can then always find an angular sampling grid $\Lambda_\omega \subset [0,\pi)$ with $|\Lambda_\omega| = \dim(\G)$ which is a set of uniqueness and
a set of sampling for $\G$.

\begin{figure}
\centering
\begin{tikzpicture}[scale=1.045]
\foreach \x in {6pt, 18pt, 24pt, 30pt, 33pt, 42pt, 58pt, 63pt}
{
  \draw[very thin,black!50] (0,0) circle (\x);
  \foreach \w in {3, 20, 26, 33, 47, 50, 58, 64, 80, 82, 85, 87, 93, 99, 103, 111, 121, 129, 134, 149, 160, 164, 169, 175, 178}
  {
    \fill (\w:\x) circle (1.5pt);
    \fill (\w+180:\x) circle (1.5pt);
  };
};
\end{tikzpicture}
\caption{Example of an irregular polar sampling grid in parallel geometry that is the Cartesian product of a radial and an angular irregular sampling grid. All circles may be unequally spaced, but the angular pattern must be the same on each circle. Only the first few circles are shown.}\label{fig:example}
\end{figure}
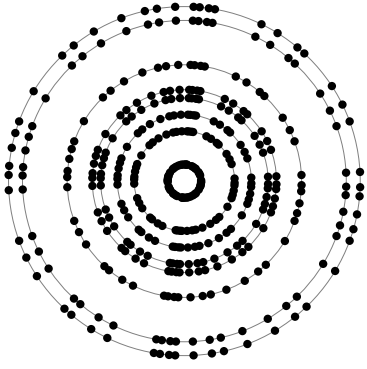
\begin{thm}[Sampling theorem for the Radon transform]\label{thm:sampledradon}
Let $\Lambda_\omega$ be a set of uniqueness and, thus, set of sampling for some finite dimensional subspace $\G \subset L^2([0,\pi])$, $\Lambda_s \subset \R$ a uniformly discrete set and let $\Lambda = \Lambda_s \times \Lambda_\omega$. Let $\H \subset \BP_{rR}(\R^2)$ be defined as $\H = (\Radon^{-1}\circ\mathcal{I})(\BP_{rR}(\R)\otimes \G)$ and let $\Radon_\Lambda : \H \to l^2(\Lambda)$ denote the sampled Radon transform.
\begin{enumerate}[(i)]
\item For $\Radon_\Lambda$ to be injective
it is sufficient that $D_{bm}(\Lambda_s) > R/\pi$.
\item For $\Radon_\Lambda$ to be continuous and continuously invertible
it is sufficient that $D_f(\Lambda_s) > R/\pi$.
\end{enumerate}
\end{thm}
\begin{proof}
Due to Theorem~\ref{thm:samplingtheorem}, the conditions are sufficient for $\Lambda_s$ being a set of uniqueness, resp.\ set of sampling, for $\PW_R(\R)$ and thus also for the subspace $\BP_{rR}(\R)$. With the assumptions on $\Lambda_\omega$, we use Theorem~\ref{thm:tensorproduct} to see that $\Lambda$ is a set of uniqueness, resp.\ set of sampling, for $\BP_{rR}(\R)\otimes \G$. Hence, the sampling operator $S_\Lambda \colon \BP_{rR}(\R) \otimes \G \to l^2(\Lambda)$ is injective, resp.\ continuous and continuously invertible on its range. These properties pass to the sampled Radon transform as all remaining operators are isomorphisms.
\end{proof}
Figure~\ref{fig:example} shows an example of an irregular sampling grid in parallel geometry that is symmetric about the origin.

\section{Conclusion}
As a consequence of the continuity of the inverse Radon transform of bandpass functions shown in Theorem~\ref{thm:radonisomorphism}, the reconstruction problem is not, strictly speaking, ill-posed, i.e., choosing $\BP_{rR}(\R^2)$ for reconstruction is stabilizing.

We will provide a reconstruction formula in an upcoming journal publication. For functions with finite angular Fourier series, the $\sinc$ function expansion is particularly suited for computation of the inverse Radon transform as all but a single one-dimensional integration can be carried out analytically using the theory of Bessel functions. To illustrate the practicality of our analytical reconstruction formula, we applied our method to the reconstruction of a non radially bandlimited image from its irregularly sampled Radon transform (Fig.~\ref{fig:recon}).
\begin{figure}
\includegraphics[scale=0.735]{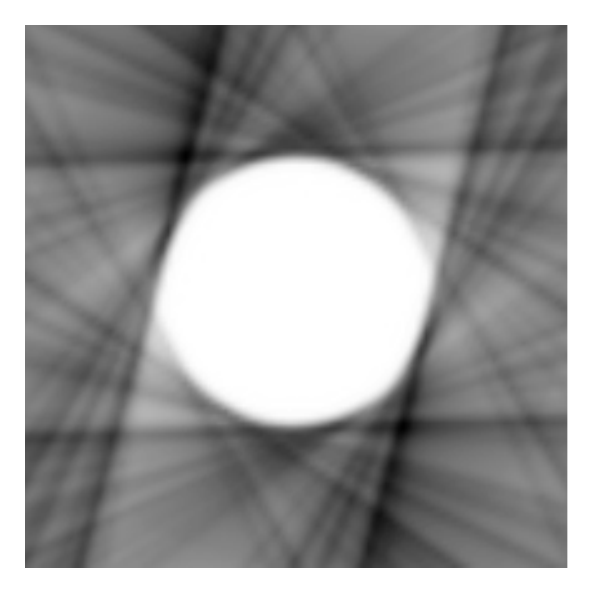}\hfill\includegraphics[scale=0.735]{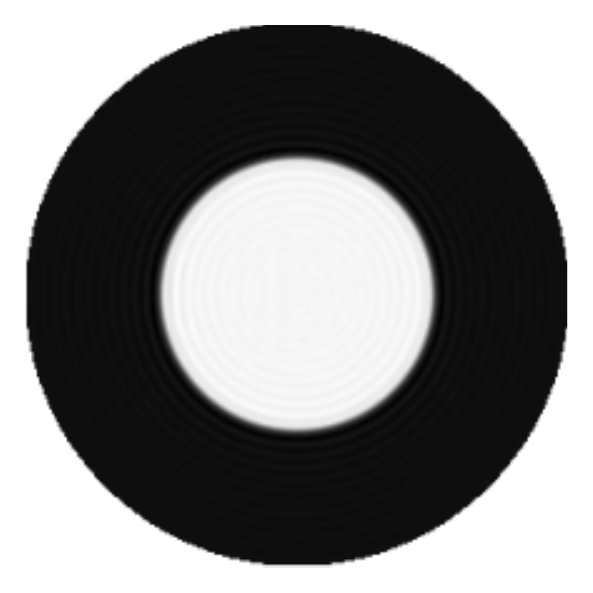}
\caption{Comparison of the reconstructions of a circle as obtained by Matlab's \texttt{iradon} function (Image Processing Toolbox), which uses filtered backprojection and a Ram-Lak filter, and a procedure that is based on sampling theory, where $\sinc$-functions for radial components and complex exponentials for angular components were mapped through the inverse Radon transform.}\label{fig:recon}
\end{figure}
\bibliographystyle{IEEEtran}
\bibliography{IEEEabrv,bibtex}

\end{document}